\documentclass{amsart}
\usepackage{cma}
\usepackage[colorlinks=true, citecolor=blue, linkcolor=red, urlcolor=blue]{hyperref}
\usepackage[T1]{fontenc}
\usepackage[utf8]{inputenc}
\usepackage[english]{babel}
\usepackage{amsfonts,amssymb,amsmath,mathrsfs,amscd}
\usepackage{graphicx}
\usepackage{array,longtable}
\usepackage{color}
\usepackage{tikz}
\usepackage{wrapfig}
\usepackage[normalem]{ulem}
\usepackage{enumitem}

\newcommand\suppress[1]{}

\newlength\wvtextpercent
\theoremstyle{plain}

\theoremstyle{definition}

\renewenvironment{proof}{\textsc{Proof.}}{\hfill$\square$}

\begin{document}
\renewcommand{\descriptionlabel}[1]{\hspace\labelsep \normalfont\bfseries #1}
\title{On Best Lacunary System in Orlicz Spaces}
\author{Olga Zavarzina}

\address{Lomonosov Moscow State University, Faculty of Mechanics and Mathematics, Leninskie Gory 1, Main Building, 119991, Moscow, Russia}
\email{olga.zavarzina@math.msu.ru}

\subjclass[2020]{46E30, 26D07}
\keywords{Khintchine inequality, Orlicz spaces, best lacunary systems, direction cosines, hypergeometric function}

\begin{abstract} Stechkin's classical results on the best lacunary system in $L_p$ spaces, given by the direction cosines on the unit sphere, are extended to Orlicz spaces $L_{\Phi}$. It is shown that for any $N$-function ${\Phi}$ the $L_{\Phi}$-norm of a linear combination of the direction cosines is completely determined by the $\ell_2$-norm of the coefficient vector. Consequently, the system is $S_{\Phi}(M)$-lacunary with a constant $M=K_{{\Phi},n} \sqrt{n}$, where $K_{{\Phi},n}$ coincides with the $L_{\Phi}$-norm of a single coordinate function. Moreover, under the additional convexity condition on ${\Phi}(\sqrt{u})$, this constant is proved to be optimal, so the direction cosines form the best lacunary system in Orlicz spaces. Explicit formulas for the constant are derived for $N$-functions ${\Phi}(u)=e^{u^2}-1$ and ${\Phi}(u)=\cosh{u}-1$, and expressed in terms of hypergeometric functions.
\end{abstract}

\maketitle

\section{Introduction}

In the theory of function spaces, the question of estimating norms of linear combinations of functions stands as one of the key problems. A classical result in this area is Khintchine inequalities established in 1923. In the work \cite{Khintchine1923}, it is proved that for any sequence of independent random variables $(\varepsilon_k)_{k\ge1}$, taking on values $\pm 1$ with equal probability (the Rademacher system), and for any $0<p<\infty$, there exist positive constants $A_p,B_p$ such that for any $a_1,\dots,a_n\in \mathbb{R}$ the following holds
\[
A_p\Big(\sum_{k=1}^n a_k^2\Big)^{1/2}
\le
\Big(\mathbb E\Big|\sum_{k=1}^n \varepsilon_k a_k\Big|^p\Big)^{1/p}
\le
B_p\Big(\sum_{k=1}^n a_k^2\Big)^{1/2}.
\]
The optimal values of constants $A_p$ and $B_p$ in this inequality for the Rademacher system were obtained by U. Haagerup \cite{Haagerup1981}. In particular, for the upper constant $B_p$, there is an explicit formula
\[
B_p=
\begin{cases}
1, & 0<p\le2,\\[6pt]
2^{1/2}\left(\dfrac{\Gamma\big((p+1)/2\big)}{\sqrt{\pi}}\right)^{1/p}, & p>2,
\end{cases}
\]
where $\Gamma$ is the gamma function.

The first significant generalization of Khintchine inequalities to the case of arbitrary independent random variables \(X_1,\dots,X_n\) with \(\mathbb{E}X_i=0\) and \(\mathbb{E}|X_i|^p<\infty\) for some $p>1$ was made by J. Marcinkiewicz and A. Zygmund \cite{Marcinkiewicz1937}. Under these conditions, there exist positive constants $A_p,B_p$, depending only on \(p\), such that
 \[
    A_p \, \mathbb{E}\!\left( \sum_{i=1}^n X_i^2 \right)^{p/2} \leq \mathbb{E}\!\left| \sum_{i=1}^n X_i \right|^p \leq B_p \, \mathbb{E}\!\left( \sum_{i=1}^n X_i^2 \right)^{p/2}.
    \]
Haskell P. Rosenthal \cite{Rosenthal1970} in 1970 developed this idea providing a more structural and precise estimate
    \[
\Big(\mathbb{E}\Big|\sum_{i=1}^n X_i\Big|^p\Big)^{1/p}
\le D_p \cdot \max\!\Big\{ \big(\sum_{i=1}^n \mathbb{E} X_i^2\big)^{1/2},
\big(\sum_{i=1}^n \mathbb{E}|X_i|^p\big)^{1/p} \Big\},
\]
   which holds for $p > 2$ and $D_p$ is a constant depending only on $p$. Another major advance in the development of the theory was finding exact values of constants in Rosenthal inequality. The first considerable result in this direction was the work of S. A. Utev \cite{Utev1985}, who showed that the constant $D_p$ for $p>4$ may be chosen as $ \| \mathcal{P}(1) \|_p$, where $\mathcal{P}(1)$ stands for the symmetrized Poisson random variable with parameter 1. For the range $2 < p \le 4$ the problem was solved in the work \cite{Figiel1997} by T. Figiel, P. Hitchenko, W. B. Johnson, G. Schechtman, and J. Zinn. In this article, it is proved that for independent symmetric random variables
$X_1,\dots,X_n$ with $2<p\le 4$, the optimal constant $D_p$ in Rosenthal inequality is
\[
D_p = \big(1 + \|g\|_p^p\big)^{1/p},
\]
where $g$ is a standard normal random variable
and $\|g\|_p = (\mathbb{E}|g|^p)^{1/p}$. 

S. B. Stechkin worked on the problem of optimality for systems of $n$ functions that satisfy the upper Khintchine inequality. In his paper \cite{Stechkin1961} devoted to the study of such lacunary systems, the following definition appears. 
\begin{definition}  Let \( p > 2 \) and \( M > 0 \). A system of linearly independent functions \( \{f_k\}_{k=1}^n \), defined on a space \( E \) with a probability measure \( \mu \), is called \textit{\( S_p(M) \)-lacunary} if \( f_k \in L^2(E) \) (\( k = 1, \ldots,n \)) and for any \( a_1, \ldots, a_n \in \mathbb{R} \) the inequality
\begin{equation}
\label{Eq:lacunary}
\left\| \sum_{k=1}^n a_k f_k \right\|_{p} \leq M \left\| \sum_{k=1}^n a_k f_k \right\|_{2},
\end{equation}
holds true, where
\[
||f||_p=\Big( \int\limits_E |f(t)|^p dt\Big)^{1/p} \quad (p>0).
\]
A system of functions $\{f_k\}_{k=1}^n$ is called \textit{the best lacunary system} for given $p>2$ and $n\in \mathbb{N}$ if the constant $M$ in \eqref{Eq:lacunary} is minimal for all possible systems of $n$ functions satisfying a similar condition.
\end{definition}
In \cite{Stechkin1961}, Stechkin also provided a construction of such a best lacunary system for any $n\in \mathbb{N}$ and an arbitrary $p> 2$. Its elements are direction cosines of the outward normal to the unit Euclidean sphere in $\mathbb{R}^n$. The idea behind the next notion can be found in Gaposhkin’s paper \cite{Gaposhkin1967}.

\begin{definition}  An infinite system of functions $\{f_k\}_{k=1}^\infty$, \( f_k \in L^2(E) \), is called an \textit{asymptotically \( S_p(M) \)-lacunary system} if inequality \eqref{Eq:lacunary} holds for all $n\in \mathbb{N}$ and \( a_1,\ldots, a_n \in\mathbb{R}\). If the constant $M$ in this case is minimal for all possible systems satisfying such a condition, then this system is called \textit{asymptotically best} for the given $p>2$.
\end{definition}

 V. F. Gaposhkin established that the Rademacher system forms an asymptotically best lacunary system in the case of $L^p$ spaces, thereby refining Stechkin's fundamental results.

 The study of lacunary systems of functions is carried out by means of functional analysis, methods of probability theory and the theory of stochastic processes. In recent decades, a number of deep results have been proved in this area, in particular, theorems by J. Bourgain. For a compact abelian group G with dual group $\Gamma$, a subset $\Lambda\subset\Gamma$ is called a $\Lambda(p)$-set (for $p>2$) if the closure of the linear span of the characters from $\Lambda$ in $L^p(G)$ coincides with the closure in $L^2(G)$. Equivalently, there exists a constant $M>0$ such that for every trigonometric polynomial $f$ with frequencies in $\Lambda$,
\[
\|f\|_{L^p(G)}\le M\|f\|_{L^2(G)}.\]
The smallest such constant is called the $\Lambda(p)$-constant of $\Lambda.$ In \cite{Bourgain1989}, Bourgain showed that for any uniformly bounded orthogonal system of $n$ functions, a subsystem of size $\sim n^\frac{2}{p} $, that is \( S_p(M) \)-lacunary with a constant $M$ depending only on $p$, can be selected. Moreover, it was proved that for any $2<p<\infty$ there exists a set $\Lambda \subset \mathbb{Z}$ which is a $\Lambda(p)$-set, but is not a $\Lambda(r)$-set for any $r>p$. An equally significant achievement was the article by S. J. Szarek \cite{Szarek1986}, where he provided a generalization of Bourgain's result to the case of vector coefficients.

The aforementioned work \cite{Figiel1997}, moreover, develops a general method for studying optimality problems in classes of Orlicz functions. As was later demonstrated by Peškir \cite{Peskir1993}, a particular case of this method yields the sharp constant $C=\sqrt{8/3}$ in the Khintchine inequality for the exponential Orlicz space generated by $\Phi(x)=e^{x^2}-1$, and it is also employed in the present work for deriving the main results.

\begin{definition}
 An \textit{$N$-function} is an even, convex and continuous function ${\Phi}:\mathbb{R}\rightarrow [0,\infty)$ such that ${\Phi}(u)>0, u\neq 0, {\Phi}(0)=0 $ and
\[
\lim_{u \to 0} \frac{{\Phi}(u)}{u} = 0, \quad \lim_{u \to \infty} \frac{{\Phi}(u)}{u} = \infty.
\]
\end{definition}
\begin{example}
(a) ${\Phi}(u)=|u|^p, p>1,$

(b) ${\Phi}(u)=e^{|u|}-1,$

(c) ${\Phi}(u)=|u|^p \ln{(1+|u|)}, p>1.$
\end{example}

\begin{definition} Let ${\Phi}$ be an $N$-function. An \textit{Orlicz space} $L_{\Phi}=L_{\Phi}(S_n)$ consists of all measurable functions $f: S_n \rightarrow \mathbb{R} $, for which there exists $\lambda = \lambda(f) > 0$ such that
\[
\int\limits_{S_n} {\Phi} \left( \frac{|f(s)|}{\lambda} \right) ds < \infty.
\]
In an Orlicz space, the Luxemburg norm is defined:
\[
\left\| f \right\|_{L_{\Phi}} = \inf \Big\{ \lambda >0: \int\limits_{S_n} {\Phi} \left( \frac{|f(s)|}{\lambda} \right) ds \leq 1 \Big\}.
\]
\end{definition}

In the present work, the best lacunary system of direction cosines from the mentioned Stechkin's work is investigated. Firstly, it is shown that the property of being lacunary originally established for $L^p$-spaces can be transferred to a broader class of Orlicz spaces $L_{\Phi}$. Specifically, the $L_{\Phi}$-norm of an arbitrary linear combination of functions of this sequence, as in the case of $L^p$, is completely determined by the $l^2$-norm of the sequence of its coefficients with a constant $K_{{\Phi},n} \sqrt{n}$. Secondly, for two important examples of $N$- functions, $\Phi(u)=e^{u^2}-1$ and $\Phi(u)=\cosh{u}-1$, the corresponding constants are obtained in explicit forms through hypergeometric equations. Thirdly, it is proved that under an additional condition of convexity of ${\Phi}(\sqrt{u})$ imposed on the $N$-function $\Phi$ the system of direction cosines remains the best lacunary system, which generalizes Stechkin's results to the case of Orlicz spaces.

\section{ Notations}

This section lists the main notations used throughout the paper.

\begin{description}
\item[$\mathbb{R}^n$] is the $n$-dimensional Euclidean space.
\item[$S_n$] is the unit sphere in $\mathbb{R}^n$.
\item[$n!$] is the factorial: $n! = 1\cdot2\cdots n$ for $n\in\mathbb{N}$, and $0! = 1$.
\item[$\Gamma(a)$] is the the gamma function.
\item[$B(a,b)$] is the beta function.
\item[$(a)_m$] is the the Pochhammer symbol:
\[
(a)_m=\dfrac{\Gamma(a+m)}{\Gamma(a)}.\]
\item[${}_1F_1(a;c;z)$] is the confluent hypergeometric function (Kummer's function):
\[
{}_1F_1(a;c;z)=\sum_{m=0}^\infty \frac{(a)_m}{(c)_m}\frac{z^m}{m!}.\]

\item[${}_0F_1(;b;z)$] is the generalized hypergeometric function:
\[{}_0F_1(;b;z)=\sum_{m=0}^\infty \frac{1}{(b)_m}\frac{z^m}{m!}.\]
\item[$L^p$] is the Lebesgue space with norm
\[
||f||_p = \Big(\int\limits_{S_n}|f|^p ds\Big)^{1/p}, 1\le p<\infty.
\]
\item[$L_{\Phi}$] is the Orlicz space generated by an $N$-function ${\Phi}$; the Luxemburg norm is defined as
\[||f||_{L_{\Phi}} = \inf\Big\{\lambda>0: \int\limits_{S_n} {\Phi}\left( \frac{|f(s)|} {\lambda}\right) ds \le 1\Big\}.
\]
\item[$\mathbb{E}$] is the expectation.
\item[$O(n)$] is the orthogonal transformation group.
\item[$\langle\cdot,\cdot\rangle$] is the inner product.
\item[$(E,\mathcal A,\mu)$] is a probability space with measure $\mu$.
\end{description}

\section{The System of Direction Cosines}
 Let \( \mathbb{R}^n \) be an \( n \)-dimensional Euclidean space, consisting of points \( x = (x_1, x_2, \ldots, x_n) \), \( S_n \) be the unit sphere in this space, given by the equation:
\[
||x||_2^2=\sum_{k=1}^n x_k^2 = 1
\]
The area of the sphere \( S_n \) is denoted by \( \mu_n \) and expressed as
\begin{equation}
\label{Eq:Square}
\mu_n = \frac{2\pi^{n/2}}{\Gamma(n/2)}
\end{equation}

On the set \(  S_n \) with the normalized Lebesgue measure \( ds=\frac{d\sigma}{\mu_n}\), we can define a system of \( n \) linearly independent functions \( \alpha_1(s), \alpha_2(s), \ldots, \alpha_n(s) \) as follows: if \( s = (x_1, x_2, \ldots, x_n) \) , then \( \alpha_k(s) = x_k \). In other words, \( \{\alpha_k(s)\}_{k=1}^n \)  is \textit{the system of direction cosines} of the outward normal \( n(s) \) to the sphere \( S_n \) at the point \( s \).

The system \( \{\alpha_k(s)\}_{k=1}^n \) is orthogonal. Moreover, the value of the integral \(\int_{S_n} \alpha_k^2(s) ds\) does not depend on \(k\). More precisely, the equality holds
\begin{equation}
    \label{Eq:Norm}
\int\limits_{S_n} \alpha_k^2(s) ds = \frac{1}{n}, \quad 1 \leq k \leq n.
\end{equation}

 For the system of direction cosines, the following theorems established by S. B. Stechkin in his work \cite{Stechkin1961} are true.

\begin{theorem}
\label{T:1}
     Let $0<p_0<p$. Then for any numbers $a_1,a_2,...,a_n$ we have the equality
\begin{equation}
\left\|\sum\limits_{k=1}^n a_k\alpha_k\right\|_{p}= \frac{K_p (n)}{K_{p_0} (n)} \left\|\sum\limits_{k=1}^n a_k\alpha_k\right\|_{p_0},
\end{equation}
 where for any $r>0$
  \begin{equation}
  K_r (n)=\left( \frac{\Gamma(\frac{r+1}{2}) \Gamma(\frac{n}{2})}{\sqrt{\pi} \Gamma (\frac{n+r}{2})}\right)^\frac{1}{r}.
  \end{equation}
\end{theorem}
\begin{theorem}
\label{T:2}
Let \(p > 2, n\in \mathbb{N}\) and $(E,\mathcal A,\mu)$ be a probability space. Suppose that for $f_k \in L^p(E) , k=1,2,...,n$, and  any numbers \(a_1, a_2, \ldots, a_n\) the inequality
\[
\left\|\sum_{k=1}^n a_k f_k \right\|_{p} \leq M_p(n) \left\|\sum_{k=1}^n a_k f_k \right\|_{2},
\]
holds. Then
\[
M_p(n) \geq \frac{K_p(n)}{K_2(n)} = K_p(n) \sqrt{n}.
\]
\end{theorem}

\section{The System of Direction Cosines in Orlicz Spaces}

We will show that Theorems \ref{T:1} and \ref{T:2} also remain valid in the case of Orlicz spaces,
which are a generalization of $L^p$-spaces.
First, we prove an auxiliary lemma.

\begin{lemma}
\label{Lem:1}
Let $S_{n} \subset \mathbb R^n$ be the unit sphere,
$ds$ be the normalized surface measure on $S_{n}$, which is
invariant under the orthogonal transformation group $O(n)$.
Let $\Theta \colon \mathbb R \to \mathbb R$ be a measurable function such that for each $a \in \mathbb R^n$
\[
I(a)
=
\int\limits_{S_{n}} \big| \Theta( \langle a,s \rangle ) \big| \,ds < \infty.
\]
Then the value $I(a)$ depends only on the length $\| a \|_2$ of the vector $a$, and not on its direction.
Specifically, for any $a\in\mathbb R^n$ we have the equality
\begin{equation}
\label{Eq:Theta}
I(a)
=
I ( \| a \|_2 e_n )
=
\int\limits_{S_{n}}
\Theta \big( \| a \|_2 \alpha_n (s) \big) \, ds,
\end{equation}
where $e_n = (0,\dots,0,1)$ is the standard basis vector.
\end{lemma}

\begin{proof}
Assume $a \ne 0$.
Since every nonzero point of $\mathbb R^n$ can be carried by an orthogonal transformation into $e_n$, there exists an orthogonal matrix $Q\in O(n)$ such that $Q^T a=\|a\|_2 e_n$.
Performing the change of variables $s=Q t$ and using invariance of the measure $ds$ under the action $Q$, we have $ds (s) = ds(t)$ and the domain of integration remains $S_{n}$. Then
\begin{align*}
\int\limits_{S_{n}} \Theta( \langle a,s \rangle ) \, ds
&= \int\limits_{S_{n}} \Theta \big( \langle a, Q t \rangle \big) \, ds(t)
= \int\limits_{S_{n}} \Theta\big( \langle  Q^T a, t \rangle \big)\,ds(t) \\
&= \int\limits_{S_{n}} \Theta \big( \langle \|a\|_2 e_n, t \rangle \big) \, ds(t)
= \int\limits_{S_{n}} \Theta \big( \|a\|_2 \, \alpha_n (t) \big)\,ds(t),
\end{align*}
which proves equality \eqref{Eq:Theta}.
\end{proof}

\begin{theorem}
\label{T:3}
Let $\Phi$ be an $N$-function. Then for any numbers $a_1,\dots,a_n$ the following holds
\begin{equation}
\label{Eq:Khintchine}
\left\| \sum_{k=1}^n a_k \alpha_k \right\|_{ L_{\Phi}(S_n) }
=
K_{{\Phi},n} \sqrt{n} \left\| \sum_{k=1}^n a_k \alpha_k \right\|_{ L^2(S_n) },
\end{equation}
where
\begin{equation}
\label{Eq:Khintchine-2}
K_{{\Phi},n} = \| \alpha_n \|_{ L_{\Phi}(S_{n}) }.
\end{equation}
\end{theorem}

\begin{proof}
By analogy with the proof of Theorem \ref{T:1} we have
\[
\int\limits_{S_n}
{\Phi} \left( \frac{\sum_{k=1}^n a_k \alpha_k(s)}{\lambda} \right) ds
=
\int\limits_{S_n} {\Phi} \left( \frac{\langle a,n(s)\rangle}{\lambda} \right) ds.
\]
Applying Lemma~\ref{Lem:1} to the integral on the right and computing the norm on the left in \eqref{Eq:Khintchine}, we obtain
\begin{equation}
\label{Eq:Khintchine-11}
\begin{split}
\left\| \sum_{k=1}^n a_k \alpha_k \right\|_{ L_{\Phi}(S_n) }
&
=
\inf \left\{ \lambda>0 \colon \int\limits_{S_n}
{\Phi} \left( \frac{||a||_2 \langle e_n,n(t)\rangle}{\lambda} \right) ds(t) \leq 1 \right\}
\\
&
=
\inf \left\{ \lambda>0 \colon \int\limits_{S_n}
{\Phi} \left( \frac{\langle e_n,n(s)\rangle}{\lambda / ||a||_2} \right) ds \leq 1\right\}
\\
&
=
\| a \|_2 \inf \left\{ \lambda>0 \colon
\int\limits_{S_n} {\Phi} \left( \frac{\langle e_n,n(s) \rangle}{\lambda} \right) ds \leq 1 \right\}
\\
&
=
\left( \sum_{k=1}^n a_k^2 \right)^{1/2} \| \alpha_n \|_{L_{\Phi}(S_n)}.
\end{split}
\end{equation}
Furthermore, from the orthogonality of the system \( \{\alpha_k \}_{k=1}^n \)
and equality \eqref{Eq:Norm} it follows that
\begin{equation}
\label{Eq:Khintchine-12}
\left\| \sum_{k=1}^n a_k \alpha_k \right\|_{L^2 (S_n)}^2
=
\sum_{k=1}^n a_k^2 \left\| \alpha_k \right\|_{L^2 (S_n)}^2
=
\frac{1}{n} \sum_{k=1}^n a_k^2.
\end{equation}
The expressions \eqref{Eq:Khintchine-11} and \eqref{Eq:Khintchine-12}
imply \eqref{Eq:Khintchine}, as it was to be shown.
\end{proof}

It is instructive to compute the constant $K_{\Phi,n}$ for some $N$-functions ${\Phi}$.
First of all, we derive a formula for the integral $\int\limits_{S_n} F ( |x_n| ) ds$,
where $F \colon [-1,1] \to \mathbb R$ is an even function. We split $S_n$ into the upper and lower hemispheres
\[
S^{\pm}_n = \left\{ (x_1,\ldots,x_n) \in S_n \colon x_n = \pm \left( 1-\sum\limits_{i=1}^{n-1} x_i^2 \right)^{\frac{1}{2}} \right\}
\]
and parametrize the upper hemisphere by
\[
X(x_1,\dots,x_{n-1})=\big(x_1,\dots,x_{n-1},\,x_n(x_1,\dots,x_{n-1})\big),\qquad
x_n(x_1,\dots,x_{n-1})=\sqrt{1-\sum_{i=1}^{n-1}x_i^2},
\]
where
\[
(x_1,\dots,x_{n-1})\in U_n=\big\{(x_1,\dots,x_{n-1}):\ \sum_{i=1}^{n-1} x_i^2\le 1\big\}.
\]
The metric tensor has components
\[
g_{ij}=\big\langle X_{x_i},X_{x_j}\big\rangle
=\delta_{ij}+\frac{\partial x_n}{\partial x_i}\frac{\partial x_n}{\partial x_j},
\]
where
\[
\frac{\partial x_n}{\partial x_i}=\frac{1}{2} \Big( 1-\sum\limits_{i=1}^{n-1} x_i^2\Big)^{-\frac{1}{2}} 2x_i=-\frac{x_i}{x_n}.
\]
Since the matrix $G=(g_{ij})$ has the form $G=I_{n-1}+v v^\top$ with
$v=(\partial x_n/\partial x_1,\dots,\partial x_n/\partial x_{n-1})^\top,$ we obtain
\[
\det(g_{ij})=1+\|v\|^2
=1+\sum_{i=1}^{n-1}\Big(\frac{x_i}{x_n}\Big)^2=\frac{x_n^2+\sum_{i=1}^{n-1}x_i^2}{x_n^2}=\frac{1}{x_n^2}.
\]
Consequently, the surface element
\[
d\sigma=\sqrt{\det(g_{ij})}\;dx_1\cdots dx_{n-1}=\frac{1}{x_n}\;dx_1\cdots dx_{n-1}
\]
and the normalized surface measure
\[
ds=\frac{d\sigma}{\mu_n},\qquad \mu_n=\frac{2\pi^{n/2}}{\Gamma(n/2)}.
\]
In particular, for an even function F
\begin{equation}
\label{Eq:F}
\int\limits_{S_n}F(|x_n|)\,ds=\frac{2}{\mu_n}\int\limits_{U_n}F\big(x_n(x_1,\ldots,x_{n-1})\big)\,\frac{1}{x_n(x_1,\ldots,x_{n-1})}\,dx_1\cdots dx_{n-1}.
\end{equation}
is true. Note that
\[
\int\limits_{U_n} \frac{F(x_n(x_1,\ldots,x_{n-1}))}{x_n(x_1,\ldots,x_{n-1})} \, dx_1\dots dx_{n-1} = 2^{n-1} \int\limits_{U_n^+} \frac{F\big(x_n(x_1,\dots,x_{n-1})\big)}{x_n(x_1,\dots,x_{n-1})} \, dx_1\dots dx_{n-1},
\]
where
\[
U_n^+ = \{ (x_1,\dots,x_{n-1}) : x_i \ge 0,\; \sum_{i=1}^{n-1} x_i^2 \le 1 \}.
\]
Perform the change $y_i = x_i^2, \quad dx_i = \frac{1}{2\sqrt{y_i}} dy_i $, so that
\begin{multline*}
\int\limits_{U_n^+} \frac{F(x_n(x_1,\ldots,x_{n-1}))}{x_n(x_1,\ldots,x_{n-1})} \, dx_1\dots dx_{n-1} = \\
=\int\limits_{\sum y_i \le 1} \frac{F\bigl(\sqrt{1 - \sum_{i=1}^{n-1} y_i}\bigr)}{\sqrt{1 - \sum_{i=1}^{n-1} y_i}} \cdot
\frac{1}{2^{n-1} \sqrt{y_1 \dots y_{n-1}}} \, dy_1 \dots dy_{n-1}.
\end{multline*}
Multiplying by $2^{n-1}$, we obtain
\begin{multline*}
\int\limits_{U_n} \frac{F(x_n(x_1,\ldots,x_{n-1}))}{x_n(x_1,\ldots,x_{n-1})} \, dx_1\dots dx_{n-1} = \\
=\int\limits_{\sum y_i \le 1} \frac{F\bigl(\sqrt{1 - \sum_{i=1}^{n-1} y_i}\bigr)}{\sqrt{1 - \sum_{i=1}^{n-1} y_i}} \cdot
\frac{1}{\sqrt{y_1 \dots y_{n-1}}} \, dy_1 \dots dy_{n-1}.
\end{multline*}
We use Liouville's formula 
\[
\int\limits_F f\Bigl(\sum_{i=1}^m y_i\Bigr) \prod_{i=1}^m y_i^{p_i-1} \, dy_1\dots dy_m =
\frac{\prod_{i=1}^m \Gamma(p_i)}{\Gamma\bigl(\sum_{i=1}^m p_i\bigr)} \int\limits_0^1 f(r) \, r^{\sum p_i - 1} \, dr.
\]
for the simplex $F = \{ y \ge 0 : \sum_{i=1}^m y_i \le 1 \}$. In our case $m = n-1$, $p_i = \frac{1}{2}$,
\[
\int\limits_{U_n} \frac{F(x_n(x_1,\ldots,x_{n-1}))}{x_n(x_1,\ldots,x_{n-1})} \, dx_1\dots dx_{n-1} = \frac{\prod_{i=1}^{n-1} \Gamma(1/2)}{\Gamma\bigl((n-1)/2\bigr)} \int\limits_0^1
\frac{F\bigl(\sqrt{1 - r}\bigr)}{\sqrt{1 - r}} \, r^{(n-1)/2 - 1} \, dr.
\]
\[
= \frac{\pi^{(n-1)/2}}{\Gamma\bigl((n-1)/2\bigr)} \int\limits_0^1
\frac{F\bigl(\sqrt{1 - r}\bigr)}{\sqrt{1 - r}} \, r^{(n-3)/2} \, dr.
\]
With the substitution $t = \sqrt{1 - r}$, $r = 1 - t^2$, $dr = -2t \, dt$, the variable $t: 1 \to 0$ as $r: 0 \to 1$. Then
\[
\int\limits_{U_n} \frac{F(x_n(x_1,\ldots,x_{n-1}))}{x_n(x_1,\ldots,x_{n-1})} \, dx_1\dots dx_{n-1}  = \frac{2\pi^{(n-1)/2}}{\Gamma\bigl((n-1)/2\bigr)} \int\limits_0^1 F(t) (1 - t^2)^{(n-3)/2} \, dt.
\]
Substituting the previous expression in \eqref{Eq:F} and taking into account that $\mu_n = \dfrac{2\pi^{n/2}}{\Gamma(n/2)}$, we get
\begin{equation}
\label{Eq:6}
\int\limits_{S_n} F(|x_n|)\,ds
= \frac{2\,\Gamma\!\big(\tfrac n2\big)}{\sqrt{\pi}\,\Gamma\!\big(\tfrac{n-1}{2}\big)}
\int\limits_{0}^{1} F(t)\,\big(1-t^2\big)^{\frac{n-3}{2}} dt.
\end{equation}

\newpage
\begin{example}
\label{Ex:1}
${\Phi}_1(u)=e^{u^2}-1$. 
\end{example} By formula \eqref{Eq:6},
where \( F(t) = e^{t^2/\lambda^2} - 1 \), we obtain
\begin{align*}
I(\lambda) :=
&
\frac{2\Gamma\left(\frac{n}{2}\right)}{\sqrt\pi\,\Gamma\!\left(\frac{n-1}{2}\right)}
\int\limits_0^1
\big(e^{t^2/\lambda^2}-1\big) \, (1-t^2)^{\frac{n-3}{2}} \, dt
\\
=
&
\frac{\Gamma(\tfrac n2)}{\sqrt\pi\,\Gamma(\tfrac{n-1}{2})}
\int\limits_0^1\big(e^{u/\lambda^2}-1\big)u^{-1/2}(1-u)^{\frac{n-3}{2}}\,du
\\
= &
\frac{\Gamma(\tfrac n2)}{\sqrt\pi\,\Gamma(\tfrac{n-1}{2})} \Big(\int\limits_0^1 e^{u/\lambda^2} u^{-1/2}(1-u)^{\frac{n-3}{2}}\,du
- \int\limits_0^1 u^{-1/2}(1-u)^{\frac{n-3}{2}}\,du \Big).
\end{align*}
The second integral on the right is expressed as
\[
\int\limits_0^1 u^{-1/2}(1-u)^{\frac{n-3}{2}}\,du
= B\Big(\frac{1}{2},\frac{n-1}{2}\Big)
= \frac{\Gamma(\tfrac12)\,\Gamma(\tfrac{n-1}{2})}{\Gamma(\tfrac n2)}.
\]
To transform the first one, we use the formula
\[
\int\limits_0^1 u^{a-1}(1-u)^{b-1} e^{z u}\,du = B(a,b)\;{}_1F_1(a;a+b;z).
\qquad
\]
Here we set
\[
a = \dfrac12, \quad b = \dfrac{n-1}{2}, \quad z = \dfrac{1}{\lambda^2}.
\]
Then
\[
\int\limits_0^1 e^{u/\lambda^2} u^{-1/2}(1-u)^{\frac{n-3}{2}}\,du
= B\Big(\tfrac12,\tfrac{n-1}{2}\Big)\;{}_1F_1\!\Big(\tfrac12;\tfrac n2;\tfrac{1}{\lambda^2}\Big).
\]
Combining the results, we obtain
\[
I(\lambda)=\frac{\Gamma(\tfrac n2)}{\sqrt\pi\,\Gamma(\tfrac{n-1}{2})} \Big( \frac{\Gamma(\tfrac12)\Gamma(\tfrac{n-1}{2})}{\Gamma(\tfrac n2)} {}_1F_1\Big(\tfrac12;\tfrac n2;\tfrac{1}{\lambda^2}\Big)-\frac{\Gamma(\tfrac12)\,\Gamma(\tfrac{n-1}{2})}{\Gamma(\tfrac n2)} \Big)
= \,{}_1F_1\Big(\tfrac12;\tfrac n2;\tfrac{1}{\lambda^2}\Big)-1.
\]
The constant \( K_{{\Phi}_1,n} \) is determined from the equation \( I( K_{{\Phi}_1,n} ) = 1 \),
which in our case implies
\[
{}_1 F_1 \Big(\tfrac12;\tfrac n2;\tfrac{1}{K_{{\Phi}_1,n}^2}\Big) = 2
\]
It remains to note that the function \( {}_1F_1(\tfrac12;\tfrac n2;z) \) increases monotonically
for $z>0$, therefore the root of the equation \( I( K_{{\Phi}_1,n} ) = 1 \) is unique.

The alternative equation that determines the constant $K_{{\Phi}_1,n}$ can be presented.
Expanding the function
\[
e^{u/\lambda^2}-1
=
\sum_{m=1}^{\infty}
\frac{u^m}{m!\,\lambda^{2m}}
\]
and substituting into \eqref{Eq:6}, we obtain
\[
I(\lambda)
=
\frac{\Gamma\left(\tfrac n2\right)}{\sqrt{\pi}\,
\Gamma\left(\tfrac{n-1}{2}\right)}
\sum_{m=1}^{\infty}
\frac{1}{m!\,\lambda^{2m}}
\int\limits_0^1
u^{m-\frac12}(1-u)^{\frac{n-3}{2}}\,du,
\]
where the integral is expressed as
\[
\int\limits_0^1
u^{m-\frac12}(1-u)^{\frac{n-3}{2}}\,du
=
B\left(m+\tfrac12,\tfrac{n-1}{2}\right)
=
\frac{\Gamma\!\left(m+\tfrac12\right)
\Gamma\left(\tfrac{n-1}{2}\right)}
{\Gamma\left(m+\tfrac n2\right)}.
\]
Finally, the value $K_{{\Phi}_1,n}$
is determined by the equality
\[
\sum_{m=1}^{\infty}
\frac{\Gamma\left(\tfrac n2\right)\,
\Gamma\left(m+\tfrac12\right)}
{\sqrt{\pi}\,m!\,
\Gamma\left(m+\tfrac n2\right)}
\,K_{{\Phi}_1,n}^{-2m}
=
1.
\]

\begin{example}
\label{Ex:2}
 \( {\Phi}_2(u)=\cosh u -1\).
\end{example}
We can apply formula \eqref{Eq:6},
where \(F(t)=\cosh(t/\lambda)-1\):
\[
I(\lambda)=\frac{2\,\Gamma(\tfrac n2)}{\sqrt\pi\,\Gamma(\tfrac{n-1}{2})}
\int\limits_0^1\big(\cosh(t/\lambda)-1\big)(1-t^2)^{\frac{n-3}{2}}\,dt.
\]
After expanding the integrand in a Taylor series we get:
\[
I(\lambda)=\frac{2\,\Gamma(\tfrac n2)}{\sqrt\pi\,\Gamma(\tfrac{n-1}{2})}
\sum_{m=1}^\infty \frac{1}{(2m)!\,\lambda^{2m}}
\int\limits_0^1 t^{2m}(1-t^2)^{\frac{n-3}{2}}\,dt.
\]
The above integral can be reduced to the beta function through the substitution \(u=t^2\)
\[
\int\limits_0^1 t^{2m}(1-t^2)^{\frac{n-3}{2}}\,dt
=\frac12\int\limits_0^1 u^{m-\tfrac12}(1-u)^{\frac{n-3}{2}}\,du
=\frac12 B\Big(m+\frac12,\frac{n-1}{2}\Big).
\]
Then
\[
I(\lambda)=\frac{\Gamma(\tfrac n2)}{\sqrt\pi\,\Gamma(\tfrac{n-1}{2})}
\sum_{m=1}^\infty \frac{B\big(m+\tfrac12,\tfrac{n-1}{2}\big)}{(2m)!\,\lambda^{2m}}=\sum_{m=1}^\infty \frac{\Gamma(\tfrac n2)\,\Gamma(m+\tfrac12)}
{\sqrt\pi\,(2m)!\,\Gamma(m+\tfrac n2)}\;\lambda^{-2m}.
\]
Using the relation
\[
(2m)! = 2^{2m} m! \,\Gamma\Big(m+\tfrac12\Big)\frac{1}{\sqrt\pi},
\]
let us simplify the coefficient and present it in the form
\[
I(\lambda)=\sum_{m=1}^\infty \frac{\Gamma(\tfrac n2)}{2^{2m} m!\,\Gamma(m+\tfrac n2)}\;\lambda^{-2m}
= \sum_{m=1}^\infty \frac{1}{m!\,(\tfrac n2)_m}\Big(\frac{1}{4\lambda^2}\Big)^m.
\]
By the definition of the generalized hypergeometric function we have
\[
I(\lambda)= {}_0F_1\Big(; \frac n2; \frac{1}{4\lambda^2}\Big)-1.
\]
The constant \(K_{{\Phi}_2,n}\) is determined from the condition \(I(K_{{\Phi}_2,n})=1\), namely from the equation
\[
{}_0F_1\!\Big(; \frac n2; \frac{1}{4K_{{\Phi}_2,n}^2}\Big)=2.
\]
Obviously, without reducing the sum to a hypergeometric function, we would obtain
\[
\sum_{m=1}^{\infty}
\frac{\Gamma\left(\tfrac n2\right)}
{2^{2m} m!\,
\Gamma\left(m+\tfrac n2\right)}
\,K_{{\Phi}_2,n}^{-2m}
=
1.
\]

\begin{theorem}
\label{T:4}
Let ${\Phi}$ be an $N$-function such that the function $\Psi(u) := {\Phi} (\sqrt u)$ is convex on $[0,\infty)$. Let $(E,\mathcal A,\mu)$ be a probability space and $\{f_k\}_{k=1}^n\subset L_{\Phi}(E)\cap L^2(E)$ be a system of linearly independent functions. Suppose there exists a constant $C$ such that for any numbers $a_1,\dots,a_n$ the inequality
\begin{equation}
\label{Eq:Khintchine-41}
\Big\| \sum_{k=1}^n a_k f_k \Big\|_{L_{\Phi}(E)}
\le
C \Big\| \sum_{k=1}^n a_k f_k \Big\|_{L_2(E)}
\end{equation}
holds.
Then $C \ge K_{{\Phi},n} \sqrt n$, where $K_{{\Phi},n}$ is defined by \eqref{Eq:Khintchine-2}.
\end{theorem}

\begin{proof}
Let $\lbrace g_1,\ldots,g_n \rbrace$ be an orthonormal system
with the same linear span as the system $\lbrace f_1,\ldots,f_n \rbrace$.
Since the linear spans of the systems coincide,
any polynomial in the system $\{ f_k \}$
is a polynomial in the system $\{ g_k \}$.
Thus, in proving \eqref{Eq:Khintchine-41}
we can assume that the system $\lbrace f_k \rbrace$ is orthonormal.

For \(t\ge0\), define the function
\[
J(t):=\int\limits_{S_n} {\Phi}\bigl(t\alpha_n(s)\bigr)\,ds.
\]
For any vector \(a\in\mathbb R^n\) with \(\|a\|_2=1\) we have
\begin{equation}
\label{Eq:Khintchine-42}
\int\limits_E {\Phi}\Big(\frac{\sum_{k=1}^n a_k g_k(x)}{C}\Big)\,d\mu(x)\le 1.
\end{equation}
We choose  \(a=\phi(s)=(\phi_1(s),\dots,\phi_n(s))\), where \(s\in S_n\),
taking into account that
\( \sum_k \alpha_k(s) g_k(x) = \langle \alpha(s), g(x) \rangle \),
and integrate inequality \eqref{Eq:Khintchine-42} over \(s\in S_n\)
\[
\int\limits_{S_n}
\int\limits_E {\Phi}\Big( \frac{ \langle \alpha(s), g(x) \rangle }{C} \Big) \, d\mu(x) \,ds \le 1.
\]
The integral over $s$ depends only on $\| g(x) \|_{l_2} =: G(x)$ by Lemma~\ref{Lem:1}.
Therefore, changing the order of integration, we obtain
\begin{equation}
\label{Eq:Khintchine-43}
\int\limits_E J \Big(\frac{G(x)}{C}\Big)\,d\mu(x) \le 1.
\end{equation}
Let us introduce an auxiliary function \(H \colon [0,\infty)\to\mathbb R\) as
  \[
    H(u):=J(\sqrt u)=\int\limits_{S_n} {\Phi} ( \sqrt{u} \alpha_n (s) ) \, ds.
  \]
By assumption, for each fixed $s$ the integrand
is convex in \( u \in [0,\infty)\),
therefore $H$ is also convex. Noting that
\begin{equation}
\label{Eq:Khintchine-44}
\int\limits_E J\Big(\frac{G(x)}{C}\Big)\,d\mu(x) = \int\limits_E H\Big(\frac{G(x)^2}{C^2}\Big)\,d\mu(x),
\end{equation}
then taking into account that
\[
\int\limits_E G(x)^2\,d\mu(x) = \sum_{k=1}^n \int\limits_E g_k(x)^2\,d\mu(x) = n,
\]
and applying Jensen's inequality, we derive
\[
    \int\limits_E H\Big(\frac{G(x)^2}{C^2}\Big)\,d\mu(x) \ge H\Big(\int\limits_E \frac{G(x)^2}{C^2}\,d\mu(x)\Big)
    = H\Big(\frac{n}{C^2}\Big).
\]
Formulas \eqref{Eq:Khintchine-43} and \eqref{Eq:Khintchine-44} yield
  \[
    H\Big(\frac{n}{C^2}\Big) \le 1.
  \]
  Then
\begin{equation}
\label{Eq:Khintchine-45}
    J\Big(\frac{\sqrt n}{C}\Big) \le 1.
\end{equation}
Using  \eqref{Eq:Khintchine-45}, we deduce that
the set $S=\{\lambda>0 \colon J(1/\lambda)\le1\}$ is non-empty.
The function ${\Phi}$ is continuous and non-decreasing on \([0,\infty)\) due to the definition of an $N$-function,
consequently, the function \(J \) also is. Accordingly, the function \( J(1/\lambda) \) is continuous and non-increasing on \((0,\infty)\). Then the set $S$ has the form of an interval $[\inf S,\infty)$. But \( \inf S = K_{\Phi} \). Hence,
\begin{equation}
\label{Eq:Khintchine-46}
    J\Big(\frac{1}{K_{\Phi}}\Big) = 1.
\end{equation}
From \eqref{Eq:Khintchine-45} and \eqref{Eq:Khintchine-46} we conclude
  \[
    \frac{\sqrt n}{C} \le \frac{1}{K_{\Phi}},\qquad C \ge K_{\Phi}\sqrt n,
  \]
this completes the proof.
\end{proof}

\begin{corollary}
For any $n \in \mathbb{N}$, the system \( \{\alpha_k \}_{k=1}^n \)
is the best lacunary system in the Orlicz space $L_{\Phi}(S_n)$ with the constant
\[
C(\{\alpha_k\}) = K_{{\Phi},n} \sqrt{n}.
\]
\end{corollary}

\begin{proof}
See Theorems~\ref{T:3} and \ref{T:4}.
\end{proof}

\end{document}